\documentclass[12pt]{article}
\usepackage{amsmath,amsfonts,amssymb,amsthm}
\newcommand{\R}{\mathbb{R}}

\newcommand{\Z}{\mathbb{Z}}
\newcommand{\Q}{\mathbb{Q}}

\newcommand{\ca}{\mathcal A}
\newcommand{\cb}{\sqrt[3]{2}}
\newcommand{\twosum}[2]{\sum_{\substack{#1\\#2}}}

\newtheorem{thm}{Theorem}[section]
\newtheorem{lem}[thm]{Lemma}
\begin{document}
\title{The largest prime factor of $X^3+2$}
\author{A.J. Irving\\
Centre de recherches math\'ematiques, Universit\'e de Montr\'eal}
\date{}

\maketitle

\begin{abstract}
Improving on a theorem of Heath-Brown, we show that if $X$ is sufficiently large  then a positive proportion of the values $\{n^3+2:n\in (X,2X]\}$ have a prime factor larger than $X^{1+10^{-52}}$.
\end{abstract}

\section{Introduction}

Given an irreducible polynomial $f$ over $\Z$, whose values have no fixed prime divisor, it is conjectured that there are infinitely many $n$ for which $f(n)$ is prime.  This problem appears to be extremely difficult when $\deg f\geq 2$ so it is natural to look for weaker statements which can be proven.  For example, rather than insisting that $f(n)$ is prime one can ask only that it has a large prime factor.  We therefore consider lower bounds for $P(x;f)$, the largest prime which divides 
$$\prod_{n\leq x}f(n).$$
The best result for polynomials of arbitrary degree is due to Tenenbaum \cite{tenenbaum} who showed that 
$$P(x;f)\gg x\exp((\log x)^{A})$$
for any $A<2-\log 4$.  However, for low degree polynomials one can do much better, gaining a power of $x$.  The first result of this kind is due to Hooley \cite{hooleylargequadratic} who showed that 
$$P(x;X^2+1)\gg x^{\frac{11}{10}},$$
which was subsequently improved by Deshouillers and Iwaniec \cite{deshiwaniec} to 
$$P(x;X^2+1)\gg x^{1.202}.$$
Hooley also considered the case of the cubic polynomial $X^3+2$.  In \cite{hooleylargecubic} he showed that, assuming certain bounds for short Kloosterman sums, we have 
$$P(x;X^3+2)\gg x^{1+\frac{1}{30}}.$$
An unconditional result of this form was established by Heath-Brown in \cite{rhbx32}, albeit with a considerably smaller exponent than $\frac{1}{30}$. Specifically,  let $\varpi=10^{-303}$ and suppose that $X$ is sufficiently large.  Heath-Brown proved that, for a positive proportion of the integers  $n\in (X,2X]$, the largest prime factor of $n^3+2$ exceeds $X^{1+\varpi}$.  In this work we will show that this conclusion holds with a larger value of $\varpi$.

\begin{thm}\label{mainthm}
Let $\varpi=10^{-52}$ and suppose that $X$ is sufficiently large.  Then, for at least a positive proportion of the integers $n\in (X,2X]$, $n^3+2$ has a prime factor in excess of $X^{1+\varpi}$.  In particular we have 
$$P(x;X^3+2)\gg x^{1+\varpi}.$$
\end{thm}

This will be proven by refining one part of Heath-Brown's argument.  Lemma \ref{lem3} gives a lower bound for a weighted sum, in which the weights depend exponentially on the number of prime factors of $n^3+2$ which exceed $X^{\frac{1}{321}}$,  and we require a lower bound for the number of nonzero terms.  Heath-Brown achieves this by taking the maximum of the weight, which is extremely large and therefore gives a very small value of $\varpi$.  Our improvement comes from observing that there are relatively few terms in the sum for which the weight is large and therefore a more efficient treatment can be given.  

The key ingredient in our work is  the estimation of the quantity $T(h,\delta)$, defined to be the number of $n\in (X,2X]$ for which $n^3+2$ has at least $h$ prime factors, when counted with multiplicities, which exceed $X^\delta$.  Our estimates will be of the form 
$$T(h,\delta)\leq X(c(h,\delta)+o(1)),$$
as $X\rightarrow \infty$, where $c(h,\delta)$ are constants which we must give explicitly.  We will prove two results of this form, Lemmas \ref{firstbound} and \ref{secondbound}.  The latter is sharper but requires a significant amount of computation to use.  We will apply our results with $\delta=\frac{1}{321}$ and $133\leq h\leq 963$.  Roughly speaking, the important fact is that for these values our bound on $c(h,\delta)$ is smaller than $2^{-h}$.  We will show that Lemma \ref{firstbound} yields such a bound when $h$ is greater than  a certain large multiple of $\log\frac{1}{\delta}$.  This means that the exponential dependence on $\delta$, described by Heath-Brown in \cite{rhbx32}, has been removed.  However, our interest is in the particular value $\delta=\frac{1}{321}$ so such asymptotic results should be treated with caution as the size of the implied constants can be more significant.  

We conjecture that the true asymptotic size of $T(h,\delta)$ should be of the above form but with smaller values of $c(h,\delta)$.  In fact, we believe that the proportion of $n$ counted by $T(h,\delta)$ should be asymptotic to the proportion of $n\in (X^3,2X^3]$ having at least $h$ prime factors in excess of $X^\delta$.  This latter quantity can be computed, see Billingsley \cite{billingsley} or Donnelly and Grimmett \cite{donnellygrimmett} for details.  When working with the values $n^3+2$ we are constrained by the available level of distribution information.  Specifically, we can only give an asymptotic formula for 
$$\#\{n\in (X,2X]:d|n^3+2\}$$
when $d\leq X$  and there is thus no hope of us establishing the asymptotic for $T(h,\delta)$.  Instead, we give upper bounds by finding sets $\mathcal D$ of $d\leq x$ such that any value of $n^3+2$ counted by $T(h,\delta)$ is divisible by at least one $d\in\mathcal D$.  Our methods could easily be extended to handle irreducible polynomials of arbitrary degree, or much more general sets having some level of distribution.  Naturally, the results for higher degree or lower level of distribution would be weaker.

It is probable that significant further improvement to the value of $\varpi$ should be possible.  
In particular, by optimising Heath-Brown's arguments one might hope for improvements to Lemma \ref{lem3}.  
By a more careful choice of the various parameters it should be possible to get a result for a larger value of $\delta$ and/or a better lower bound for the sum. In addition, the results of this paper might be improved, either by giving better combinatorial constructions or simply by giving better estimates for the integrals needed for  Lemmas \ref{firstbound} and \ref{secondbound}.

In a recent preprint \cite{dartyge}, Dartyge has extended Heath-Brown's methods to handle the quartic polynomial $X^4-X^2+1$.  She showed that for sufficiently large $x$ we have 
$$P(x;X^4-X^2+1)\geq x^{1+\varpi}$$
with $\varpi=10^{-26531}$.  It seems probable that our methods could be used to give a considerable improvement to this exponent.

\subsection*{Acknowledgements}

This work was completed whilst I was a CRM-ISM postdoctoral fellow at the Universit\'e de Montr\'eal.

\section{Heath-Brown's Approach}

As in \cite{rhbx32} we will work with the subset of algebraic integers in $\Q(\cb)$ given by 
$$\ca=\{n+\cb:X<n\leq 2X\}.$$
By the next lemma, the elements of $\ca$ are composed entirely of first degree prime ideals so Theorem \ref{mainthm} will follow if we can show that a positive proportion of them have a prime ideal factor $P$ with $N(P)\geq X^{1+\varpi}$.  If $I$ is an ideal we let 
$$\ca_I=\{\alpha\in \ca:I|\alpha\}$$
and 
$$\rho(I)=\#\{n\pmod {N(I)}:n\equiv\cb\pmod I\}.$$
Heath-Brown gives the following result which describes the function $\rho$.

\begin{lem}[{\cite[Lemma 1]{rhbx32}}]
If $n\equiv\cb\;\mod{I}$ is solvable with a rational integer $n$, then
$I$ is composed of first degree prime ideals only.  Moreover $I$
cannot be divisible by two distinct prime ideals of the same norm, nor
by $P_{2}^{2}$ or $P_{3}^{2}$, where $P_{2}$ and $P_{3}$ are the
primes above $2$ and $3$ respectively.
In all other cases the congruence is solvable, and we have
$\rho(I)=1$.   Moreover, if $I$ is an ideal for which $\rho(I)=1$,
then for any $m\in\Z$, we have $I|m$ if and only if $N(I)|m$.
\end{lem}

If we define 
$$\log^{(1)}(n^3+2)=\twosum{P^e||(n+\cb)}{N(P)\leq 3X}\log N(P^e)$$
then Heath-Brown proves the following.

\begin{lem}[{\cite[Lemma 2]{rhbx32}}]\label{lem2}
Suppose $\alpha,\delta>0$ and that we can find at least $\alpha X$ elements $n+\cb\in \ca$ for which 
\begin{equation}\label{requiredbound}
\log^{(1)}(n^3+2)\geq (1+\delta)\log X.
\end{equation}
The number of $n+\cb\in \ca$ which have a prime ideal factor $P$ with $N(P)\geq X^{1+\alpha\delta/2}$ is then at least $(\delta\alpha^2+o(1))X$.  
\end{lem}

As shown by Heath-Brown, a sufficient condition for (\ref{requiredbound}) is that $n+\cb$ has an ideal factor $J=KL$ with 
\begin{equation}\label{Jbound}
X^{1+\delta}< N(KL)\leq X^{1+2\delta}
\end{equation}
and 
\begin{equation}\label{Kbound}
X^{3\delta}<N(K)\leq X^{4\delta}.
\end{equation}
It therefore remains to find a lower bound for the proportion of elements of $\ca$ divisible by such an ideal.  We will let $K$ run over the set $\mathcal K$ of  degree $1$ prime ideals satisfying (\ref{Kbound}).  We will then let $L$ run over a set $\mathcal L(K)$ of ideals which satisfy (\ref{Jbound}). 

As in \cite{rhbx32} we will use a lower-bound sieve to restrict to ideals $L$ which have no prime ideal factor of small norm.  Specifically we let $\lambda_d$ denote the lower-bound linear sieve of level $X^{3\delta}$ and 
$$Q=\prod_{p<X^\delta}p,$$
the product being restricted to primes which split in $\Q(\cb)$. We consider the sum 
$$S=\sum_{K\in \mathcal K}\sum_{L\in \mathcal L(K)}(\sum_{d|(Q,N(L)}\lambda_d)\#\ca_{KL}.$$ 
By definition of the sieve we have 
\begin{eqnarray*}
S&\leq&\sum_{K\in \mathcal K}\twosum{L\in \mathcal L(K)}{(N(L),Q)=1}\#\ca_{KL}\\
&=&\sum_{n+\cb\in\ca}\#\{(k,l):K\in \mathcal K,L\in\mathcal L(K),(N(L),Q)=1,KL|n+\cb\}.\\
\end{eqnarray*}
It follows that any $n+\cb$ counted with positive weight in $S$ must satisfy (\ref{requiredbound}).  The remainder of the work \cite{rhbx32} is concerned with the estimation of the sum $S$ from below.  The following lemma summarises the conclusion.

\begin{lem}\label{lem3}
Let $\delta=\frac{1}{321}$.  The sets of ideals $\mathcal L(K)$ can then be chosen in such a way that (\ref{Jbound}) is satisfied and we have 
$$S\geq \left(9.2\times 10^{-8}+o(1)\right)X.$$
\end{lem}

Let
$$W(n+\cb)=\#\{(k,l):K\in \mathcal K,L\in\mathcal L(K),(N(L),Q)=1,KL|n+\cb\}.$$
We wish to give a lower bound for 
$$\#\{n+\cb\in\ca:W(n+\cb)>0\}$$
whereas the last lemma gives an estimate for 
$$\sum_{n+\cb\in\ca}W(n+\cb).$$
Heath-Brown's approach is to estimate the maximum of $W(n+\cb)$ and use that 
$$\#\{n+\cb\in \ca:W(n+\cb)>0\}\geq\frac{1}{\max W(n)}\sum_{n+\cb\in \ca}W(n).$$
This is inefficient since the maximum is exponentially large in terms of $\delta$ but it is only achieved on a very low density subset of $\ca$.  Our improvement is therefore to show that terms with large $W(n+\cb)$ give a small contribution to $S$.  We let $\Omega_\delta(n+\cb)$ be the number of prime ideals $P$, counted with multiplicities, for which $P|n+\cb$ and $N(P)\geq X^\delta$.  Observe that, when $X$ is large enough, we have 
$$\Omega_\delta(n+\cb)\leq [3/\delta].$$
For any given $n+\cb$, the number of prime ideals $P|n+\cb$ with $N(P)\geq X^{3\delta}$, and therefore the number of choices for $K$,  may be bounded by $\min(\Omega_\delta(n+\cb),[1/\delta])$.  For each $K$ the number of possible $L$ is at most $2^{\Omega_\delta(n+\cb)}$ so we may conclude that 
$$W(n+\cb)\leq \min(\Omega_\delta(n+\cb),[1/\delta])2^{\Omega_\delta(n+\cb)}.$$
Let $H$ be a parameter to be chosen later.  We have 
\begin{eqnarray*}
\lefteqn{\#\{n+\cb\in \ca:W(n+\cb)>0\}}\\
\hspace{1cm}&\geq&\#\{n+\cb\in \ca:W(n+\cb)>0,\Omega_\delta(n+\cb)\leq H\}\\
\hspace{1cm}&\geq&\frac{2^{-H}}{\min(H,[1/\delta])}\twosum{n+\cb\in\ca}{\Omega_\delta(n+\cb)\leq H}W(n+\cb)\\
\hspace{1cm}&=&\frac{2^{-H}}{\min(H,[1/\delta])}\left(\sum_{n+\cb\in\ca}W(n+\cb)-\twosum{n+\cb\in\ca}{\Omega_\delta(n+\cb)>H}W(n+\cb)\right).\\
\end{eqnarray*}
It therefore remains to give an upper bound for 
$$\twosum{n+\cb\in\ca}{\Omega_\delta(n+\cb)>H}W(n+\cb).$$
This is at most 
$$\sum_{h>H}\min(h,[1/\delta])2^h\#\{n+\cb\in\ca:\Omega_\delta(n+\cb)=h\}\leq \sum_{h>H}\min(h,[1/\delta])2^hT(h,\delta)$$
where
$$T(h,\delta)=\#\{n+\cb\in\ca:\Omega_\delta(n+\cb)\geq h\}.$$

Since $n+\cb$ is composed of first degree prime ideals we may take norms to deduce that $\Omega_\delta(n+\cb)$ is equal to the number of rational prime factors $p|n^3+2$, counted with multiplicities, for which $P\geq X^\delta$.  We will denote this latter quantity by $\Omega_\delta(n^3+2)$, so that 
$$T(h,\delta)=\#\{n:X<n\leq 2X:\Omega_\delta(n^3+2)\geq h\}.$$

\section{First Estimate for $T(h,\delta)$}

We wish to bound the number of $n\in (X,2X]$ for which $n^3+2$ has at least $h$ prime factors, when counted with multiplicities, which exceed $X^\delta$.  Specifically, for a fixed real $\delta>0$ and integer $h$, we aim to find an explicit constant $c(h,\delta)$ such that 
\begin{equation}\label{Trequiredbound}
T(h,\delta)\leq X(c(h,\delta)+o(1)).
\end{equation}
We will achieve this by constructing a set $\mathcal D$ of integers with the following properties.

\begin{enumerate}
\item If $n\in (X,2X]$ is such that $\Omega_\delta(n^3+2)\geq h$ then $n^3+2$ is divisible by some $d\in\mathcal D$.

\item If $d\in\mathcal D$ then $d=O(X)$.

\item If $d\in\mathcal D$ then all the prime factors of $d$ exceed $X^\delta$.  
\end{enumerate}

Given such a set $\mathcal D$ we have 
$$T(h,\delta)\leq \sum_{d\in\mathcal D}A_d$$
where
$$A_d=\#\{n\in (X,2X]:n^3+2\equiv 0\pmod d\}.$$
Defining the arithmetic function $\nu$ by 
$$\nu(d)=\#\{n\pmod d:n^3+2\equiv 0\pmod d\}$$
we have 
$$A_d=\frac{X\nu(d)}{d}+O(\nu(d)).$$
By the Chinese Remainder Theorem $\nu$ is multiplicative.  In addition, for any prime power $p^e$ we have $\nu(p^e)\ll 1$.  Since any $d\in\mathcal D$ has at most $[1/\delta]$ prime factors we conclude that $\nu(d)\ll_\delta 1$ and therefore that 
$$A_d=\frac{X\nu(d)}{d}+O_\delta(1).$$
The cardinality of $\mathcal D$ cannot exceed the total number of integers up to $O(X)$ with no prime factor smaller than $X^\delta$.  This latter quantity is $O_\delta(X/\log X)$ so we conclude that 
$$T(h,\delta)\leq X\sum_{d\in \mathcal D}\frac{\nu(d)}{d}+O_\delta(X/\log X).$$
Since $\nu$ is not completely multiplicative it is convenient to deal trivially with the $d$ which are not squarefree.  If $\delta>\frac{1}{2}$ then all $d\in\mathcal D$ are squarefree.  Otherwise we have 
\begin{eqnarray*}
\twosum{d\in\mathcal D}{\mu(d)=0}\frac{\nu(d)}{d}&\ll &\twosum{d\in \mathcal D}{\mu(d)=0}\frac{1}{d}\\
&\leq&\sum_{X^\delta\leq p\ll \sqrt{X}}\twosum{d\in\mathcal D}{p^2|d}\frac{1}{d}\\
&\leq&\sum_{X^\delta\leq p\ll \sqrt{X}}\twosum{d\leq X}{p^2|d}\frac{1}{d}\\
&=&\sum_{X^\delta\leq p\ll \sqrt{X}}\sum_{d\leq X/p^2}\frac{1}{dp^2}\\
&\ll&\log X\sum_{X^\delta\leq p\ll \sqrt{X}}\frac{1}{p^2}\\
&\ll&X^{-\delta}\log X.\\
\end{eqnarray*}
We therefore have
$$T(h,\delta)\leq X\left(\twosum{d\in \mathcal D}{\mu(d)\ne 0}\frac{\nu(d)}{d}+o(1)\right).$$
The next lemma gives a construction of a suitable set $\mathcal D$.

\begin{lem}
Suppose $\delta\in (0,1)$ and $h\geq 3$.  Let $k=[h/3]$ and let $\mathcal D=\mathcal D(h,\delta)$ be the set of integers 
$$d=p_1p_2\ldots p_k$$
with 
$$X^\delta\leq p_1\leq p_2\leq\ldots\leq p_k$$
and 
$$p_1p_2\ldots p_{k-1}p_k^{h-k+1}\leq 9X^3.$$
The set $\mathcal D$ then satisfies the above three hypotheses.
\end{lem}

\begin{proof}
If $d\in\mathcal D$ then, by the assumptions on its factorisation, we have 
$$d^3=p_1^3p_2^3\ldots p_k^3\leq p_1\ldots p_{k-1}p_k^{2k+1}\leq p_1p_2\ldots p_{k-1}p_k^{h-k+1}\ll X^3$$
and thus $d\ll X$.  By construction, the prime factors of $d$ all exceed $X^\delta$ and therefore we have verified the second and third hypotheses.  

Suppose $n\in (X,2X]$ and $\Omega_\delta(n^3+2)\geq h$.  Write 
$$X^\delta\leq p_1\leq p_2\leq\ldots\leq p_h$$
for the smallest $h$ prime factors of $n^3+2$ which are at least $X^\delta$.  We then have 
$$p_1p_2\ldots p_{k-1}p_k^{h-k+1}\leq p_1p_2\ldots p_h\leq n^3+2\leq 9X^3$$
so that $d=p_1p_2\ldots p_k\in\mathcal D$.  This verifies the first hypothesis and therefore completes the proof.
\end{proof}

We must now estimate, for the $\mathcal D$ constructed in the last lemma, the sum 
$$\twosum{d\in \mathcal D}{\mu(d)\ne 0}\frac{\nu(d)}{d}.$$
We begin with the well-known estimate (see for example Diamond and Halberstam \cite[Proposition 10.1]{diamondhalberstam}) 
$$\sum_{p\leq x}\frac{\nu(p)\log p}{p}=\log x+O(1).$$
If $d\in\mathcal D$ with $\mu(d)\ne 0$ then $\nu(d)=\nu(p_1)\nu(p_2)\ldots \nu(p_k)$ so we may repeatedly apply partial integration and summation to deduce that 
$$\twosum{d\in \mathcal D}{\mu(d)\ne 0}\frac{\nu(d)}{d}=\int_T\frac{dt_1\ldots dt_k}{t_1\ldots t_k(\log t_1)\ldots (\log t_k)}+o(1)$$
where 
$$T=\{(t_1,\ldots,t_k)\in\R^k:X^\delta\leq t_1\leq t_2\leq \ldots\leq t_k,t_1\ldots t_{k-1}t_k^{h-k+1}\leq 9X^3\}.$$
We then make the substitution $t_i=X^{s_i}$ to deduce that 
$$\twosum{d\in \mathcal D}{\mu(d)\ne 0}\frac{\nu(d)}{d}=\int_{R(h,\delta)}\frac{ds_1\ldots ds_k}{s_1\ldots s_k}+o(1)$$
where 
$$R(h,\delta)=\{(s_1,\ldots,s_k)\in\R^k:\delta\leq s_1\leq s_2\leq\ldots\leq s_k,s_1+\ldots+s_{k-1}+(h-k+1)s_k\leq 3\}.$$
Note that one first obtains the condition 
$$s_1+\ldots+s_{k-1}+(h-k+1)s_k\leq 3+\frac{\log 9}{\log X}$$
but the error in removing the $\frac{\log 9}{\log X}$ is $o(1)$.  

Unfortunately it appears to be very difficult to evaluate the above integral exactly.  However, since the integrand is positive we may produce an upper bound by enlarging the domain of integration.  If $s_1,\ldots,s_{k-1}\geq \delta$ and 
$$s_1+\ldots+s_{k-1}+(h-k+1)s_k\leq 3$$
then 
$$s_k\leq \frac{3-(k-1)\delta}{h-k+1}.$$
Since the integrand is symmetric under permutations of the coordinates we may therefore deduce that 
$$\int_{R(h,\delta)}\frac{ds_1\ldots ds_k}{s_1\ldots s_k}\leq \frac{1}{k!}\int_{[\delta,\frac{3-(k-1)\delta}{h-k+1}]^k}\frac{ds_1\ldots ds_k}{s_1\ldots s_k}=\frac{1}{k!}\left(\log\frac{3-(k-1)\delta}{h-k+1}-\log \delta\right)^k.$$
Combining the above we see that we have proved the following.

\begin{lem}\label{firstbound}
Suppose $\delta\in (0,1)$ and $h\geq 3$.  If $k=[h/3]$ then 
$$T(h,\delta)\leq X\left(\frac{1}{k!}\left(\log\frac{3-(k-1)\delta}{(h-k+1)\delta}\right)^k+o(1)\right).$$
\end{lem}

This lemma gives (\ref{Trequiredbound}) with 
\begin{eqnarray*}
c(h,\delta)&=&\frac{1}{k!}\left(\log\frac{3-(k-1)\delta}{(h-k+1)\delta}\right)^k\\
&\leq&\frac{1}{k!}\left(\log\frac{9}{2h\delta}\right)^k\\
&\ll&\frac{1}{\sqrt h}\left(\frac{e}{k}\log\frac{9}{2h\delta}\right)^k,\\
\end{eqnarray*}
where the last inequality follows from Stirling's approximation.  If we now suppose that $\delta\geq \exp(-\eta h)$, for some small $\eta>0$, then 
$$\frac{e}{k}\log\frac{9}{2h\delta}\leq \frac{e}{k}(\log\frac{9}{2}+\eta h-\log h)\ll \eta.$$
We deduce that there exists an absolute constant $A$ such that if $h$ is sufficiently large in terms of $\eta$ then 
$$\frac{e}{k}\log\frac{9}{2h\delta}\leq A\eta$$
and hence 
$$c(h,\delta)\ll \frac{1}{\sqrt h}(A\eta)^{[h/3]}\ll_\eta \frac{1}{\sqrt{h}}((A\eta)^{\frac{1}{3}})^h.$$
We therefore choose $\eta$ so that $(A\eta)^{1/3}<\frac{1}{2}$ and conclude that when 
$$\delta\geq \exp(-\eta h)$$
(that is when $h\geq -\frac{1}{\eta}\log \delta$)
we have 
$$c(h,\delta)\ll  \frac{2^{-h}}{\sqrt{h}}.$$
We conclude  that when $h$ is larger than a certain multiple of $-\log \delta$ we have obtained a bound for $c(h,\delta)$ which is better than $2^{-h}$.  Recall that this is roughly the type of estimate we require for our application.  

It is clear that the proof of Lemma \ref{firstbound} could be modified to handle any irreducible polynomial $f$ over $\Z$.  The quantity $k$ would then be given by $[h/\deg f]$ and the resulting bound would be 
$$T(h,\delta)\leq X\left(\frac{1}{k!}\left(\log\frac{\deg f-(k-1)\delta}{(h-k+1)\delta}\right)^k+o(1)\right).$$

\section{Second Estimate for $T(h,\delta)$}

Suppose $n\in (X,2X]$ is such that $\Omega_\delta(n^3+2)\geq h$ and let 
$$X^\delta\leq p_1\leq p_2\leq \ldots\leq p_h$$
be the $h$ smallest prime factors of $n^3+2$ which exceed $X^\delta$.  In the last section we counted such a $n$ by using   that $p_1\ldots p_{[h/3]}|n^3+2$ and $p_1\ldots p_{[h/3]}\ll X$.  In this section we will exploit the fact that, for many $n$, this product is actually much smaller than $X$ and therefore we can include more primes in it.  Specifically, let $K\in [[h/3],h-1]$ be a parameter to be chosen later and define 
$$k(n^3+2)=\max\{k\leq K:p_1\ldots p_k\leq 3X\}.$$
We showed in the last section that $k(n^3+2)\geq [h/3]$ for all $n$.  We will choose a set $\mathcal D$ which contains all the products $p_1p_2\ldots p_{k(n^3+2)}$ and therefore we derive inequalities satisfied by the primes in these products.  

Firstly, keeping the above notation, we observe that we must have 
$$p_1\ldots p_{k(n)-1}p_{k(n)}^{h-k(n)+1}\leq 9X^3.$$
Secondly, if $k(n)<K$ then 
$$p_1\ldots p_{k(n)+1}\geq 3X.$$
However, we know that 
$$p_1\ldots p_{k(n)}p_{k(n)+1}^{h-k(n)}\leq 9X^3$$
so that
$$p_{k(n)+1}\leq \left(\frac{9X^3}{p_1\ldots p_{k(n)}}\right)^{\frac{1}{h-k(n)}}.$$
We therefore conclude that 
$$p_1\ldots p_{k(n)}\geq 3X\left(\frac{9X^3}{p_1\ldots p_{k(n)}}\right)^{-\frac{1}{h-k(n)}}$$
which simplifies to 
$$\left((p_1\ldots p_{k(n)}\right)^{\frac{h-k(n)-1}{h-k(n)}}\gg_{h,k(n)} X^{\frac{h-k(n)-3}{h-k(n)}}.$$
It follows, since $k(n)<h-1$, that
\begin{equation}\label{maximal}
p_1\ldots p_{k(n)}\gg_{h,k(n)} X^{\frac{h-k(n)-3}{h-k(n)-1}}.
\end{equation}
We can now construct a suitable set $\mathcal D$.

\begin{lem}
Suppose $\delta\in (0,1)$ and $h\geq 3$.  For an integer $k\in [[h/3],K]$ let $\mathcal D_k$ be the set of 
$$d=p_1\ldots p_k\leq 3X$$
with 
$$X^\delta\leq p_1\leq p_2\leq \ldots\leq p_k,$$
$$p_1\ldots p_{k-1}p_k^{h-k+1}\leq 9X^3$$
and  
$$p_1\ldots p_{k(n)}\gg X^{\frac{h-k(n)-3}{h-k(n)-1}}.$$
In the final condition the implied constants are equal to those from (\ref{maximal}).  The condition is omitted if $k=K$.  The set 
$$\mathcal D=\bigcup_{k=[h/3]}^K\mathcal D_k$$
then satisfies the three conditions from the start of the previous section.
\end{lem}

We must now estimate 
$$\twosum{d\in \mathcal D}{\mu(d)\ne 0}\frac{\nu(d)}{d}=\sum_{k=[h/3]}^K\twosum{d\in\mathcal D_k}{\mu(d)\ne 0}\frac{\nu(d)}{d}.$$
Converting the sums to integrals, as in the last section, this is equal to 
$$\sum_{k=[h/3]}^KI(h,\delta,k)+o(1)$$
where 
$$I(h,\delta,k)=\int_{R(h,\delta,k)}\frac{ds_1\ldots ds_k}{s_1\ldots s_k}$$
and 
\begin{eqnarray*}
R(h,\delta,k)&=&\{s_1,\ldots,s_k\in \R^k:\delta\leq s_1\leq s_2\leq\ldots\leq s_k,s_1+\ldots+s_k\leq 1,\\
&&s_1+\ldots+s_{k-1}+(h-k+1)s_k\leq 3,s_1+\ldots+s_k\geq \frac{h-k-3}{h-k-1}\}.\\
\end{eqnarray*}
As in the lemma, the final inequality is omitted when $k=K$. 

Unfortunately the resulting integrals are now even harder to deal with than those occurring in the previous section. We are therefore forced to weaken many of the constraints on $R_k$ in order to reach an integral which can be evaluated.  In particular the inequality 
$$s_1+\ldots+s_k\leq 1$$
will be ignored and 
$$s_1+\ldots+s_{k-1}+(h-k+1)s_k\leq 3$$
will be replaced by the weaker 
$$s_k\leq \frac{3-(k-1)\delta}{h-k+1}.$$
as in the previous section we therefore deduce that 
$$I(h,\delta,K)\leq \frac{1}{K!}\left(\log\frac{3-(K-1)\delta}{(h-K+1)\delta}\right)^K.$$ 
If $k<K$ then we handle  the constraint 
$$s_1+\ldots+s_k\geq \frac{h-k-3}{h-k-1}$$
by multiplying the integrand by 
$$\exp\left(\alpha\left(s_1+\ldots+s_k-\frac{h-k-3}{h-k-1}\right)\right)$$
for some $\alpha>0$.  This factor is always positive and it exceeds $1$ in the region of interest.  We therefore conclude that 
\begin{eqnarray*}
I(h,\delta,k)&\leq &\frac{1}{k!}\int_{[\delta,\frac{3-(k-1)\delta}{h-k+1}]^k}\exp\left(\alpha\left(s_1+\ldots+s_k-\frac{h-k-3}{h-k-1}\right)\right)\frac{ds_1\ldots s_k}{s_1\ldots s_k}\\
&=&\frac{\exp\left(-\alpha\frac{h-k-3}{h-k-1}\right)}{k!}\left(\int_\delta^{\frac{3-(k-1)\delta}{h-k+1}}\exp(\alpha s)\frac{ds}{s}\right)^k.\\
\end{eqnarray*}
The final integral cannot be expressed using of elementary functions but it is an exponential integral which can easily be evaluated by a computer.  The value of $\alpha$ should be chosen to minimise the expression.

The conclusion of this section is the following estimate for $T(h,\delta)$.

\begin{lem}\label{secondbound}
Suppose $\delta\in (0,1)$ and $h\geq 3$.  For any integer $K\in [[h/3],h-1]$ and any positive reals $\alpha_k$ we have 
\begin{eqnarray*}
T(h,\delta)&\leq& X\left(\sum_{k=[h/3]}^{K-1}\frac{\exp\left(-\alpha_k\frac{h-k-3}{h-k-1}\right)}{k!}\left(\int_\delta^{\frac{3-(k-1)\delta}{h-k+1}}\exp(\alpha_k s)\frac{ds}{s}\right)^k\right.\\
&&\hspace{1cm}\left.+\frac{1}{K!}\left(\log\frac{3-(K-1)\delta}{(h-K+1)\delta}\right)^K+o(1)\right).
\end{eqnarray*}
\end{lem}

In practice we have found that when applying this lemma the largest contribution to the sum comes from $k=[h/3]$.  We therefore take $K$ sufficiently large so that the bound on $I(h,\delta,K)$ is smaller than this, little can be gained by taking a larger $K$.     

\section{Proof of Theorem \ref{mainthm}}

Recall that we need to estimate 
$$\sum_{h>H}\min(h,[1/\delta])2^hT(h,\delta)$$
when $\delta=\frac{1}{321}$.  This sum is not infinite as $T(h,\delta)=0$ for $h>963$.  We begin by using Lemma \ref{firstbound} to obtain 
$$\sum_{h\geq 190}\min(h,[1/\delta])2^hT(h,\delta)\leq X\left(9.2\times 10^{-10}+o(1)\right).$$
For smaller $h$ we use Lemma \ref{secondbound} to obtain 
$$\sum_{h=133}^{189}\min(h,[1/\delta])2^hT(h,\delta)\leq X\left(3.6\times 10^{-8}+o(1)\right).$$ 
Specifically, this is obtained using the value $K=[h/3]+20$ and optimising the choices of $\alpha_k$ in each instance.  We therefore choose $H=132$ and conclude that
$$\sum_{h>H}\min(h,[1/\delta])2^hT(h,\delta)\leq X\left(3.7\times 10^{-8}+o(1)\right).$$
We deduce from this that 
\begin{eqnarray*}
\#\{n+\cb\in\ca:W(n)>0\}&\geq& \frac{2^{-132}}{132}X\left(9.2\times 10^{-8}-3.7\times 10^{-8}+o(1)\right)\\
&\geq &X\left(7.7\times 10^{-50}+o(1)\right).\\
\end{eqnarray*}
We may therefore apply Lemma \ref{lem2} with $\delta=\frac{1}{321}$ and $\alpha=7.7\times 10^{-50}$ to deduce that Theorem \ref{mainthm} holds when 
$\varpi\leq \frac{\alpha\delta}{2}=1.2\times 10^{-52}$, so in particular it holds when $\varpi\leq 10^{-52}$.

\addcontentsline{toc}{section}{References} 
\bibliographystyle{plain}
\bibliography{../biblio}

\bigskip
\bigskip

Centre de recherches math\'ematiques,

Universit\'e de Montr\'eal,

Pavillon Andr\'e-Aisenstadt,

2920 Chemin de la tour, Room 5357,

Montr\'eal (Qu\'ebec) H3T 1J4

\bigskip 

{\tt alastair.j.irving@gmail.com}
 
\end{document}